\documentclass[11pt,reqno]{amsart}

\usepackage{hyperref}
\usepackage{color}
\usepackage{amssymb}

\usepackage{graphicx, comment}
\usepackage{enumerate}
\usepackage{ytableau}
\usepackage{comment}
\usepackage{amsmath}
\usepackage{mathtools}
\usepackage{graphicx}
\usepackage{tikz}
\usetikzlibrary{matrix,positioning}

\newtheorem{Theorem}{Theorem}
\newtheorem{Proposition}[Theorem]{Proposition}

\newtheorem{Lemma}[Theorem]{Lemma}

\theoremstyle{remark}
\newtheorem*{Remark}{Remark}

\newtheorem*{rems}{Remarks} 
\newenvironment{Remarks}{\begin{rems}\normalfont}{\end{rems}}

\numberwithin{equation}{section}

\allowdisplaybreaks 

\newcommand{\pqrfac}[3]{{\left({#1};#3\right)_{#2}}}

\newcommand{\elliptictheta}[1]{\theta\!\left({#1} ; p\right) }

\newcommand{\ellipticqrfac}[2]{{\left({#1}; q, p\right)_{#2}}} 

\def\ta{\theta}

\newcommand{\qrfac}[2]{{\left({#1}; q\right)_{#2}}} 

\author[G.~Bhatnagar]{Gaurav Bhatnagar
}

\address{Ashoka University, Sonipat, Haryana 131029, India}
\email{bhatnagarg@gmail.com}

\author[A.~Kumari]{Archna Kumari}
\address{Department of Mathematics, Indian Institute of Technology, Delhi 110067, India.}
\email{arcyadav856@gmail.com}

\author[M.\ J.\ Schlosser]{Michael J.\ Schlosser}
\address{Fakult\"at f\"ur Mathematik\\
Universit\"at Wien\\
Oskar-Morgenstern-Platz~1\\
A-1090 Vienna, Austria}
\email{michael.schlosser@univie.ac.at}

\title[An esoteric identity with many parameters]{An esoteric identity with many parameters and
other elliptic extensions of elementary identities }
 


\begin{document}

\begin{abstract}
We provide elliptic extensions of elementary identities such as the sum of the first $n$ odd or even numbers, the geometric sum and the sum of the first $n$ cubes.  Many such identities, and their $q$-analogues, are indefinite sums, and can be obtained from telescoping. So we used telescoping in our study to find elliptic extensions of these identities. In the course of our study, we obtained an identity with many parameters, which appears to be new even in the $q$-case. 
In addition, we recover some $q$-identities due to Warnaar. 
\end{abstract}

\maketitle

\section{Introduction}\label{intro}
The geometric sum is a staple of high school algebra. It can be written as
\begin{equation}\label{geo}
\sum_{k=0}^{n-1} q^k = \frac{1-q^{n}}{1-q} =:[n]_q,
\end{equation}
where $[n]_q$ denotes the $q$-number of $n$. 
This notation is justified because the limit as $q\to 1$ is
$$\sum_{k=0}^{n-1} 1 = n.$$
More generally, we can define $[z]_q:=(1-q^z)/(1-q)$ for any complex $z$
and observe that $\lim_{q\to 1}[z]_q=z$. 
Thus, we call $[z]_q$ the $q$-analogue of $z$.


The objective of this paper is to extend several classical and elementary identities to the so-called elliptic numbers---which are even more general than
the $q$-numbers---defined in \cite{SY2016b}. Rather surprisingly, these lead to new identities even in the $q$-case. This work is in the context of the
 rapidly developing field of elliptic combinatorics. Some recent references are  
 ~\cite{BCK2020a,   BCK2020b, DB2018, 
 BKS2023a, BGR2010, HKKS2023, MS2007, MS2020b, MS2023, SY2015,SY2016b, SY2016a,  SY2017b, SY2017a,  SY2018,  SY2021}.

To be able to define an elliptic number, we need some notation. 
The {\bf modified Jacobi theta function} of the complex number $a$
with (fixed) nome $p$ is defined as
\begin{equation*} \elliptictheta{a} := \prod\limits_{j=0}^{\infty} (1-ap^j) (1-p^{j+1}/a)
\end{equation*}
where $a\neq 0$ and $|p|<1$. 
When the nome $p=0$, the modified theta function $\elliptictheta{a}$ reduces to $(1-a)$.
We use the shorthand notation
\begin{equation*}
\elliptictheta{a_1, a_2, \dots, a_r} := \elliptictheta{a_1} \elliptictheta{a_2}\cdots  \elliptictheta{a_r}.
\end{equation*}

The elliptic analogue of a complex number $z$ is defined by \cite{SY2016b} as
\begin{subequations}
\begin{equation}\label{elliptic-n}
[z]_{a,b;q,p}:=\frac{\ta(q^z,aq^z,bq^2,a/b;p)}{\ta(q,aq,bq^{z+1},aq^{z-1}/b;p)}.
\end{equation}
This has additional (complex) parameters $a$ and $b$, in addition to the {\em base} $q$ and nome $p$. Note that 
$[0]_{a,b;q,p}=0$ and $[1]_{a,b;q,p}=1$. 
Let the elliptic weight be defined by
\begin{equation}\label{big-W}
W_{a,b;q,p}(k):=\frac{\ta(aq^{2k+1},bq,bq^2,aq^{-1}/b,a/b;p)}{\ta(aq,bq^{k+1},bq^{k+2},aq^{k-1}/b,aq^k/b;p)}q^k,
\end{equation}
for any $k$.
By the Weierstra{\ss} addition formula for theta functions (see \eqref{addf}, below) we have
\begin{equation}\label{recur1}
[x+y]_{a,b;q,p}=[x]_{a,b;q,p}+W_{a,b;q,p}(x)[y]_{aq^{2x},bq^x;q,p}.
\end{equation}
\end{subequations}
Note that if we set $p=0$ and subsequently take $a = 0$ and then $b = 0$, the elliptic weight in \eqref{big-W} reduces to $q^k$. In this case \eqref{recur1} reduces to the recurrence relation
$$[x+y]_{q}=[x]_{q}+q^x [y]_{q}.$$
This, along with the initial conditions $[0]_q=0$ and $[1]_q=1$, is used to define the $q$-number for integers. Thus, the elliptic number is indeed an extension of the $q$-number $[x]_q$ for any complex $x$.

The analogue of the geometric sum \eqref{geo}---obtained by iterating \eqref{recur1}---is as follows. (Here $n$ is assumed to be a non-negative integer.)
\begin{equation}
 1+W_{a,b;q,p}(1)+W_{a,b;q,p}(2)+\cdots+W_{a,b;q,p}(n-1)=[n]_{a,b;q,p}. \label{basic-g}
\end{equation}

Many such elementary results (even in the $q=1$ case) are examples of indefinite sums, and can be proved by telescoping, as has been shown in \cite{GB2011}. This motivates the study of elliptic extensions using these techniques. In doing so, we naturally came across the following result, which is somewhat esoteric, but appears to be new even in the $q$-case.
At this point, we would like to emphasize that the parameters in our identities
should  be chosen such that not-removable singularities and poles are avoided,
so that the identities make sense.  
\begin{Theorem}\label{big-theorem}
For any non-negative integer $n$ and complex numbers $c$, $d$, $g$ and $h$,
we have the following identity:
\begin{align}\label{bigid}
\sum_{k=0}^n&\Bigg(\frac{\big[2(gk+c)(hk+d)\big]_{a,b;q,p}\,
\big[2ghk+ch+dg\big]_{aq^{2(gk-g+c)(hk+d)},bq^{(gk-g+c)(hk+d)};q,p}}
{\big[2cd\big]_{a,b;q,p}\,
\big[ch+dg\big]_{aq^{2(c-g)d},bq^{(c-g)d};q,p}}\notag\\
&\times\prod_{j=0}^{k-1}
\frac{\big[(gj+g+c)(hj+d)\big]_{aq^{2(gj-g+c)(hj+d)},bq^{(gj-g+c)(hj+d)};q,p}}
{\big[(gj+g+c)(hj+d)\big]_{aq^{2(gj+g+c)(hj+2h+d)},
bq^{(gj+g+c)(hj+2h+d)};q,p}}\notag\\
&\times\prod_{j=0}^{k-1}
W_{aq^{2(gj+c)(hj+h+d)},bq^{(gj+c)(hj+h+d)};q,p}\big(2ghj+2gh+ch+dg\big)^{-1}\Bigg)\notag\\
={}&\frac{\big[(gn+c)(hn+h+d)\big]_{a,b;q,p}\,
\big[(g+c)d\big]_{aq^{2(c-g)d},bq^{(c-g)d};q,p}}
{\big[2cd\big]_{a,b;q,p}\,
\big[ch+dg\big]_{aq^{2(c-g)d},bq^{(c-g)d};q,p}}\notag\\
&\times
\prod_{j=1}^n\frac{\big[(gj+g+c)(hj+d)\big]_{aq^{2(gj-g+c)(hj+d)},bq^{(gj-g+c)(hj+d)};q,p}}
{\big[(gj+c)(hj-h+d)\big]_{aq^{2(gj+c)(hj+h+d)},bq^{(gj+c)(hj+h+d)};q,p}}\notag\\
&\times\prod_{j=1}^n
W_{aq^{2(gj-g+c)(hj+d)},bq^{(gj-g+c)(hj+d)};q,p}\big(2ghj+ch+dg\big)^{-1}\notag\\
&-\frac{\big[(c-g)d\big]_{a,b;q,p}\,
\big[c(d-h)\big]_{aq^{2c(h+d)},bq^{c(h+d)};q,p}}
{\big[2cd\big]_{a,b;q,p}\,
\big[ch+dg\big]_{aq^{2(c-g)d},bq^{(c-g)d};q,p}}\,
W_{aq^{2(c-g)d},bq^{(c-g)d};q,p}\big(ch+dg\big).
\end{align}
\end{Theorem}

The ``hypergeometric version'' of \eqref{bigid} is given by
\begin{align*}
\sum_{k=0}^n \frac{(gk+c)(hk+d)(2ghk+ch+dg)}{cd(ch+dg)}
= \frac{(gn+c)(hn+h+d)(gn+g+c)(hn+d)}{2cd(ch+dg)}\cr
-\frac{(d-h)(c-g)}{2(ch+dg)}.
\end{align*}
Note that this extends the well-known formula for the sum of the first $n$ cubes. Multiply both sides by $cd(ch+dg)/2$ and then take $c=d=0$, and 
$h=g=1$, to obtain
$$\sum_{k=0}^n k^3 =\bigg( \frac{n(n+1)}{2}\bigg)^2.$$
This is indeed an elementary identity, but 
its extension given in \eqref{bigid}
involves some rather unusual factors. 
Note, for example, the product
$$
\prod_{j=0}^{k-1}
\big[(gj+g+c)(hj+d)\big]_{aq^{2(gj-g+c)(hj+d)},bq^{(gj-g+c)(hj+d)};q,p}
$$
appearing with index $k$ in the sum. The associated $q$-product
(obtained by first letting $p\to 0$, followed by $a\to 0$ and $b\to 0$)
$$
t(k):=\prod_{j=0}^{k-1}
\big[(gj+g+c)(hj+d)\big]_q
$$
is rather unusual as it is not a $q$-hypergeometric term.
In particular, the ratio $t(k+1)/t(k)$ of this product, that is, 
$[(gk+g+c)(hk+d)]_q=(1-q^{(gk+g+c)(hk+d)})/(1-q)$, is not a rational function
in $q^k$; it is a rational function in $q^{k^2}$ and $q^k$, and
contains quadratic powers of $q$.

Nevertheless, \eqref{bigid} contains various extensions of well-known
elementary identities.  The following identities appear as special cases. 
\begin{subequations}
\begin{gather}\label{spc-4(i)}
\sum_{k=1}^n q^{n-k} \frac{[2k]_{q}}{[2]_{q}}  = \begin{bmatrix}n+1\\2\end{bmatrix}_{q}; 
\\
\label{spc-4(ii)}
\sum_{k=1}^n q^{n^2-k^2+n-k} \frac{[2k^2]_{q}[2k]_{q}}{[2]_{q}^2} =\bigg( \frac{[n(n+1)]_q}{[2]_q}\bigg)^2. 
\end{gather}
\end{subequations}
Here we have used the notation
$$\begin{bmatrix}n+1\\2\end{bmatrix}_{q} = \frac{[n]_q[n+1]_q}{[2]_2}.$$
The first of these is a $q$-analogue of the sum of the first $n$ natural numbers; the second is a $q$-analogue of the sum of the first $n$ cubes, which is
equivalent to a formula of Cigler~\cite[Theorem~1, $q\mapsto q^2$]{Cigler2014}.

We now provide some background information and list some notation used in this paper. 
\subsection*{Background information} \ 
\begin{enumerate}
\item 
Two important properties of the modified theta function are
\cite[Equation~(11.2.42)]{GR90}
\begin{subequations}
\begin{gather}\label{GR11.2.42}
\elliptictheta{a} =\elliptictheta{p/a} =-a\elliptictheta{1/a},
\end{gather}
and \cite[p.~451, Example~5]{WW1996}
\begin{gather}\label{addf}
\elliptictheta{xy,x/y,uv,u/v}-\elliptictheta{xv,x/v,uy,u/y}=\frac uy\,\elliptictheta{yv,y/v,xu,x/u}.
\end{gather}
\end{subequations}
This last formula is called the Weierstra{\ss} addition formula. This formula is used extensively in this paper. 
\item The following general theorem serves as a justification of
  referring to $[z]_{a,b;q,p}$, defined in \eqref{elliptic-n}, as an
  ``elliptic number''. 
\begin{Proposition}[{\cite[Theorem~1.3.3]{HR2016-lectures}}]\label{prop}
Let $g(x)$ be an elliptic function, that is, a~doubly periodic meromorphic function of the complex variable~$x$.
Then $g(x)$ is of the form:
\begin{gather*}
g(x)= \frac{\elliptictheta{ a_1q^x,a_2q^x,\dots,a_{r}q^x}} {\elliptictheta{b_1q^x,b_2q^x,\dots,b_rq^x}} c,
\end{gather*}
where $c$ is a constant, and
\begin{gather*}
a_1a_2\cdots a_{r}=b_1b_2\cdots b_r.
\end{gather*}
\end{Proposition}
This last condition is the {\em elliptic balancing condition}. If we write $q=e^{2\pi i\sigma}$, $p=e^{2\pi i\tau}$, with complex $\sigma$, $\tau$, then $g(x)$ is indeed doubly periodic in $x$ with periods $\sigma^{-1}$ and $\tau\sigma^{-1}$. 
\item Using Proposition~\ref{prop}, it is easy to see that elliptic number
  $[z]_{a,b;q,p}$ is elliptic in $z$, and also elliptic in $\log_q a$ and in
  $\log_q b$.
  
\item Similarly, the elliptic weight function $W_{a,b;q,p}(k)$ is elliptic
  in $\log_q a$, $\log_q b$ and $k$ (regarded as a complex variable). 

\item The following useful properties readily follow from the definitions.
\begin{enumerate} 
\item For any $k$ and $l$,
$\displaystyle
W_{a,b;q,p}(k+l)=W_{a,b;q,p}(k)W_{aq^{2k},bq^k;q,p}(l).
$
\item 
$W_{a,b;q,p}(0)=1$, and for any $k$,
$
W_{a,b;q,p}(-k)=W_{aq^{-2k},bq^{-k};q,p}(k)^{-1}.
$
\item For any $x$,

$\displaystyle
[-x]_{a,b;q,p}=-W_{a,b;q,p}(-x)[x]_{aq^{-2x},bq^{-x};q,p}
=-W_{aq^{-2x},bq^{-x};q,p}(x)^{-1}[x]_{aq^{-2x},bq^{-x};q,p}.
$
\item For any $x$ and $y$,
$\displaystyle
[xy]_{a,b;q,p}=[x]_{a,b;q,p}[y]_{a,bq^{1-x};q^x,p}.
$
\item For any $r$, $x$ and $y$,
\begin{align}\label{quad-rel}
[x]_{a,b;q,p}[y]_{aq^{2r+2x-2y},bq^{r+x-y};q,p}-
[x+r]_{a,b;q,p}[y-r]_{aq^{2r+2x-2y},bq^{r+x-y};q,p}&\notag\\
=[r+x-y]_{a,b;q,p}[r]_{aq^{2x},bq^x;q,p}W_{aq^{2r+2x-2y},bq^{r+x-y};q,p}(y-r).&
\end{align}
\end{enumerate}
 The property \eqref{quad-rel} is a consequence of the Weierstra{\ss}
  addition formula in \eqref{addf}.

\item In \S\ref{sec3}, we require the notation of $q$-rising factorials and their elliptic analogues. 
We define the {\em $q$-shifted factorials}, for $k=0,1,2,\dots$, as
\begin{gather*}
\qrfac{a}{k} :=\prod\limits_{j=0}^{k-1} \big(1-aq^j\big),
\end{gather*}
and for $|q|<1$,
\begin{gather*}
\qrfac{a}{\infty} := \prod\limits_{j=0}^{\infty} \big(1-aq^j\big).
\end{gather*}
The parameter $q$ is called the {\em base}. With this definition, we can write the {modified Jacobi theta function}  as
\begin{gather*} \elliptictheta{a} = \pqrfac{a}{\infty}{p} \pqrfac{p/a}{\infty}{p},\end{gather*}
where $a\neq 0$ and $|p|<1$. We define the {\em $q, p$-shifted factorials} (or {\em theta shifted factorials}), for $k$ an integer, as
\begin{gather*}
\ellipticqrfac{a}{k} := \prod\limits_{j=0}^{k-1} \elliptictheta{aq^j}.
\end{gather*}
When the nome $p=0$, 
$\ellipticqrfac{a}{k}$ reduces to $ \pqrfac{a}{k}{q}$.

We use the shorthand notations
\begin{gather*}
\ellipticqrfac{a_1, a_2,\dots, a_r}{k} := \ellipticqrfac{a_1}{k} \ellipticqrfac{a_2}{k}\cdots \ellipticqrfac{a_r}{k},\\
\qrfac{a_1, a_2,\dots, a_r}{k} := \qrfac{a_1}{k} \qrfac{a_2}{k}\cdots \qrfac{a_r}{k}.
\end{gather*}

\item 
Most of the proofs of the theorems in this paper use the following technique, explained in detail in \cite[Theorem 3.3]{GB2011}.
\begin{Lemma}[Euler's telescoping lemma]\label{E-telescoping-lemma}
 Let $u_k$, $v_k$ and $t_k$ be three sequences, such that
$$t_k=u_k-v_k.$$ 
Then we have:
\begin{equation}\label{eq:E-telescoping-lemma}
\sum_{k=0}^n \frac{t_k}{t_0}\frac{u_0u_1\cdots u_{k-1}}{v_1v_2\cdots v_{k}}
= \frac{u_0}{t_0}\left( \frac{u_1u_2\cdots u_{n}}{v_1v_2\cdots v_{n}} - \frac{v_0}{u_0}\right),
\end{equation}
provided none of the denominators in \eqref{eq:E-telescoping-lemma} are zero.
\end{Lemma}
\end{enumerate}

\subsection*{Some important specializations of the elliptic numbers and elliptic weights} It is helpful to explicitly write out some important special cases of the elliptic numbers and the elliptic weights. These cases correspond to $p=0$ (the ``$a,b;q$-case''); $p=0$ and $b\to 0$ (the ``$a;q$-case''); and, $p=0$ and $a\to 0$ (the ``$(b;q)$-case'').
 
The three special cases of the elliptic numbers are
\begin{subequations}
\begin{align} 
[z]_{a,b;q} &=\frac{(1-q^z)(1-aq^z)(1-bq^2)(1-a/b)} 
{(1-q)(1-aq)(1-bq^{z+1})(1-aq^{z-1}/b)}; \label{z-abq}\\
[z]_{a;q} &=\frac{(1-q^z)(1-aq^z)}
{(1-q)(1-aq)}q^{1-z}; \label{z-aq} \\
[z]_{(b;q)} &=\frac{(1-q^z)(1-bq^2)}
{(1-q)(1-bq^{z+1})}   \label{z-bq},
\end{align}
\end{subequations}
and called $a,b;q$-numbers, $a;q$-numbers, and $(b;q)$-numbers, respectively.
(We place parentheses in ``$(b;q)$-numbers'' but none in
``$a;q$-numbers'', to avoid confusion between the two special cases. 
This follows the notation used in \cite{SSU2021}.)

The corresponding special cases for the elliptic weight $W_{a,b;q,p}(k)$ are as follows:
\begin{subequations}
\begin{align} 
W_{a,b;q}(k) &=\frac{(1-aq^{2k+1})(1-bq)(1-bq^2)(1-aq^{-1}/b)(1-a/b)}
{(1-aq)(1-bq^{k+1})(1-bq^{k+2})(1-aq^{k-1}/b)(1-aq^k/b)}q^k; \label{W-abq}\\
W_{a;q}(k) &=\frac{(1-aq^{2k+1})}
{(1-aq)}q^{-k};\label{W-aq}\\
W_{(b;q)}(k) &=\frac{(1-bq)(1-bq^2)}
{(1-bq^{k+1})(1-bq^{k+2})}q^k  \label{W-bq}.
\end{align}
\end{subequations}

This paper is organized as follows. In Section~\ref{sec2}, we use Euler's telescoping lemma to find elliptic 
extensions of three elementary identities and discuss some interesting special cases. In Section~3, we consider elliptic extensions of several elementary identities that are obtained in an
analogous way to the $q$-identities previously obtained by one of us in \cite{MS2004}. Finally, in Section~4, we give the proof of Theorem~\ref{big-theorem} (achieved by combining
Lemma~\ref{E-telescoping-lemma} with the difference equation
\eqref{quad-rel}), and explicitly state a few noteworthy special cases.

\section{Elementary Examples}\label{sec2}
%
%

The purpose of this section is to extend three elementary identities to corresponding identities containing elliptic numbers. For each of these elliptic identities, we give some special cases for illustration.
The three identities are:
\begin{subequations}
\begin{equation}\label{sumofodd}
\sum_{k=0}^{n-1} (2k+1) = n^2;
\end{equation}
\begin{equation} \label{rising-fact}
\sum_{k=1}^n k(k+1)\cdots (k+m-1) = \frac{1}{m+1}\left( n(n+1)\cdots (n+m)\right);
\end{equation}
\begin{equation}\label{rising-fact2}
\sum_{k=1}^n \frac{1}{k(k+1)\cdots (k+m)} =\frac{1}{m}
\left( \frac{1}{m!} - \frac{1}{(n+1)(n+2)\cdots (n+m)}\right),
\end{equation}
\end{subequations}
where $m=1, 2, 3, \dots$.

First, we give an elliptic extension of the sum of the first $n$ odd integers.
\begin{Theorem} For $n$ a non-negative integer, we have 
\begin{equation}
\sum_{k=0}^{n} W_{a,b;q,p}(k)\bigg([k+1]_{a,b;q,p}[2]_{aq^{2k},bq^{k};q,p}-1\bigg)
=W_{a,b;q,p}(1)[n+1]_{a,b;q,p}[n+1]_{aq^2,bq;q,p}.\label{tel-c}
\end{equation}
\end{Theorem}
\begin{proof}
We apply Lemma~\ref{E-telescoping-lemma} and take
\begin{gather*}
u_k =  [k+1]_{a,b;q,p}[k+1]_{aq^2,bq;q,p};\\
v_k = u_{k-1} =[k]_{a,b;q,p}[k]_{aq^2,bq;q,p}=[k+1]_{a,b;q,p}[k-1]_{aq^2,bq;q,p}+W_{aq^2,bq;q,p}(k-1).\\
\intertext{Thus}
t_k =u_k-v_k = W_{aq^2,bq;q,p}(k-1)\bigg([k+1]_{a,b;q,p}[2]_{aq^{2k},bq^{k};q,p}-1\bigg)
\text{ and } 
t_0=u_0-v_0= 1.
\end{gather*}
We thus obtain \eqref{eq:E-telescoping-lemma} with these choices of
$u_k$, $v_k$ and $t_k$. Multiplication of both sides of the
identity by $W_{a,b;q,p}(1)$ gives the result. 
\end{proof}

\begin{Remark}
The elliptic analogue of $n$, namely, $[n]_{a,b;q,p}$
contains extensions of $n^2$ and of $\binom{n+1}2$,
besides other extensions. Take $z=n$ in \eqref{z-aq},
the $a;q$-number of $n$.
For $a\to\infty$ this reduces to $[n]_q$,
for $a=1$ to $([n]_q)^2 q^{1-n}$, and for $a=q$ to
$q^{1-n}[n]_q[n+1]_q/[2]_q$.
That is, the telescoping sum over odd elliptic numbers
also extends a sum over odd squares, and to a sum over
binomial coefficients. The examples in this section illustrate
some of the possibilities to obtain interesting identities by specialization.
\end{Remark}

\subsection*{Special cases of \eqref{tel-c} }

\begin{enumerate} 
\item Three immediate specializations of \eqref{tel-c} are as follows. 
\begin{itemize}
\item For the $a,b;q$-analogue, take $p=0$. 
\begin{equation}\label{tel-c-a-b}
\sum_{k=0}^{n}  W_{a,b;q}(k)([k+1]_{a,b;q} [2]_{aq^{2k}, bq^{k};q}-1)
=W_{a,b;q}(1) [n+1]_{a,b;q}[n+1]_{aq^2,bq;q}.
\end{equation}
This identity has two parameters, $a$ and $b$, in addition to the base $q$. 
\item For the $a;q$-analogue, take ${b \to 0}$ or ${b \to \infty}$
  in \eqref{tel-c-a-b}. This gives
\begin{equation}\label{tel-c-a}
\sum_{k=0}^{n}  W_{a;q}(k)([k+1]_{a;q} [2]_{aq^{2k};q}-1)
=W_{a;q}(1) [n+1]_{a;q}[n+1]_{aq^2;q}.
\end{equation}
\item For the $(b;q)$-analogue, take ${a \to 0}$ or ${a \to \infty}$
  in \eqref{tel-c-a-b}. This gives 
\begin{equation}\label{tel-c-b}
\sum_{k=0}^{n}  W_{(b;q)}(k)([k+1]_{(b;q)} [2]_{(bq^{k};q)}-1)
=W_{(b;q)}(1) [n+1]_{(b;q)}[n+1]_{(bq;q)}.
\end{equation}
\end{itemize}

\item We further specialize $a$ and $b$ to obtain two new $q$-analogues of \eqref{sumofodd}. 
\begin{itemize}
\item Take ${a \to \infty}$ in \eqref{tel-c-a}, or ${b \to 0}$ in
  \eqref{tel-c-b}, to get
\begin{equation*}\label{sp1}
\sum_{k=0}^{n}  q^{k-1}([2]_{q}[k+1]_{q}-1)
= [n+1]_{q}^2.
\end{equation*}
\item When ${a \to 0}$ in \eqref{tel-c-a}, or $b\to\infty$ in \eqref{tel-c-b}, to obtain
\begin{equation*}\label{sp2}
\sum_{k=0}^{n}  q^{2n-2k}([2]_{q}[k+1]_{q}-q^{k+1})
= [n+1]_{q}^2.
\end{equation*}

\end{itemize}
\item Take $a\to 1$ in \eqref{tel-c-a}, respectively, $b\to 1$ in \eqref{tel-c-b}, to obtain the following pair of identities: 
\begin{gather*}
\sum_{k=0}^{n}  q^{2n-2k}\bigg([2]_{q}[k+1]_{q}^2[2k+2]_{q}-q^{k+1}[2k+1]_{q}\bigg)
= [n+1]_{q}^3[n+3]_{q};\\
\sum_{k=0}^{n}  \frac{q^{k-1}}{[k+1]_{q}[k+2]_{q}}\bigg(\frac{[k+1]_{q}[2]_{q}^2}{[k+3]_{q}}-1\bigg)
= \frac{[n+1]_{q}^2}{[n+2]_{q}[n+3]_{q}}.
\end{gather*}

\item Next, take $a\to q$ in \eqref{tel-c-a}, respectively, $b\to q$ in \eqref{tel-c-b}, to obtain the following pair of identities:
\begin{gather*}
\sum_{k=0}^{n}   q^{2n-2k}\bigg([k+1]_{q}[k+2]_{q}[2k+3]_{q}-q^{k+1}[2k+2]_{q}\bigg)
= \frac{[n+1]_{q}^2[n+2]_{q}[n+4]_{q}}{[2]_{q}};\\
\sum_{k=0}^{n}  \frac{q^{k-1}}{[k+2]_{q}[k+3]_{q}}\bigg(\frac{[k+1]_{q}[2]_{q}[3]_{q}}{[k+4]_{q}}-1\bigg)
= \frac{[n+1]_{q}^2}{[n+3]_{q}[n+4]_{q}}.
\end{gather*}
\end{enumerate}

%

Next, we give an elliptic extension of \eqref{rising-fact}.

\begin{Theorem} For $n, m$ non-negative integers, we have 
\begin{align}
\sum_{k=0}^n & W_{{a,b;q,p}}(k)[m+1]_{aq^{2k},bq^k;q,p}
 \Big( [k+1]_{a,b;q,p}[k+2]_{a,b;q,p}\dots [k+m]_{a,b;q,p}\Big) \cr
&=[n+1]_{a,b;q,p}[n+2]_{a,b;q,p}\dots [n+m+1]_{a,b;q,p} \label{tel-a}.
\end{align}
\end{Theorem}
\begin{proof}
We apply Lemma~\ref{E-telescoping-lemma} and take
\begin{align*}
u_k &=  [k+1]_{a,b;q,p}[k+2]_{a,b;q,p}\dots [k+m+1]_{a,b;q,p};\\
v_k = u_{k-1} &= [k]_{a,b;q,p}[k+1]_{a,b;q,p}\dots [k+m]_{a,b;q,p},\\
\intertext{so that,} 
t_k &= W_{a,b;q,p}(k) [m+1]_{aq^{2k},bq^k;q,p} \Big( [k+1]_{a,b;q,p}[k+2]_{a,b;q,p}\dots [k+m]_{a,b;q,p}
\Big).
\end{align*}
With these substitutions, we have \eqref{eq:E-telescoping-lemma} which immediately gives us
\eqref{tel-a}.
\end{proof}

We take $m=1$ and shift the index $k\mapsto k-1$ and replace $n$ by $n-1$ in \eqref{tel-a}, to get the elliptic analogue of the sum of first $n$ even integers.
\begin{equation}\label{Elliptic-sumofeven}
\sum_{k=1}^{n}  W_{a,b;q,p}(k-1)[2]_{aq^{2k-2},bq^{k-1};q,p} [k]_{a,b;q,p} 
= [n]_{a,b;q,p}[n+1]_{a,b;q,p}.
\end{equation}
This can be regarded to be an elliptic extension of the formula for the sum of the first $n$ natural numbers:
\begin{equation}\label{triangular}
1+2+3+\cdots+n = \frac{n(n+1)}2.
\end{equation}
We list further special cases of the elliptic analogue of this elementary identity below. 

\subsection*{Special cases of \eqref{Elliptic-sumofeven}}
\begin{enumerate}
\item For the $a,b;q$-analogue, take $p=0$.
\begin{equation}\label{even-a-b-q}
\sum_{k=1}^{n}  W_{a,b;q}(k-1)[2]_{aq^{2k-2}, bq^{k-1};q}[k]_{a,b;q} 
=[n]_{a,b;q}[n+1]_{a,b;q}.
\end{equation}
\item For the $a;q$-analogue, take ${b \to 0}$ or ${b \to \infty}$ in
  \eqref{even-a-b-q}. This gives
\begin{equation}\label{even-a-q}
\sum_{k=1}^{n}  W_{a;q}(k-1)[2]_{aq^{2k-2};q} [k]_{a;q} 
= [n]_{a;q}[n+1]_{a;q}.
\end{equation}
\item For the $(b;q)$-analogue, take ${a \to 0}$ or ${a \to \infty}$ in
  \eqref{even-a-b-q}. This gives
\begin{equation}\label{even-b-q}
\sum_{k=1}^{n}  W_{(b;q)}(k-1)[2]_{(bq^{k-1};q)} [k]_{(b;q)} 
= [n]_{(b;q)}[n+1]_{(b;q)}.
\end{equation}
\item Two  $q$-analogues of \eqref{triangular}
\begin{itemize}
\item  Take the limit ${a \to \infty}$ in \eqref{even-a-q}, or ${b \to 0}$ in \eqref{even-b-q}:
\begin{equation*}
\sum_{k=1}^{n}  q^{k-1}[k]_{q} 
= \begin{bmatrix}n+1\\2\end{bmatrix}_{q}.
\end{equation*}
\item A $q$-analogue due to Warnaar \cite[Eq.\ 2]{SOW2004}: Take the limit ${a \to 0}$ in \eqref{even-a-q}, or ${b \to \infty}$ in \eqref{even-b-q}.
\begin{equation*}
\sum_{k=1}^{n} q^{2n-2k}[k]_{q}
=\begin{bmatrix}n+1\\2\end{bmatrix}_{q}.
\end{equation*}
\end{itemize}
\item Some assorted $q$-analogues.
\begin{itemize}
\item A $q$-analogue of the formula for the sum of cubes due to Warnaar~\cite[Eq.\ 2]{SOW2004}: take ${a \to 1}$ in \eqref{even-a-q}.
\begin{equation}\label{warnaar-cubes}
\sum_{k=1}^{n} q^{2n-2k}\frac{[k]_{q}^2[2k]_{q}}{[2]_{q}}
=\begin{bmatrix}n+1\\2\end{bmatrix}_{q}^2.
\end{equation}

\item Take ${b \to 1}$ in \eqref{even-b-q}.
\begin{equation*}
\sum_{k=1}^{n} q^{k-1} \frac{[2]_{q}}{[k+1]_{q}[k+2]_{q}}
=  \frac{[n]_{q}}{[n+2]_{q}}.
\end{equation*}
\item Take ${a \to q}$ in \eqref{even-a-q}.
\begin{equation*}
\sum_{k=1}^{n} q^{2n-2k}[k]_{q}[k+1]_{q}[2k+1]_{q}
=\frac{[n]_{q}[n+1]_{q}^2[n+2]_{q}}{[2]_{q}}.
\end{equation*}
\item Take ${b \to q}$ in \eqref{even-b-q}.
\begin{equation*}
\sum_{k=1}^{n} q^{k-1} \frac{[2]_{q}^2[k]_{q}}{[k+1]_{q}[k+2]_{q}[k+3]_{q}}
= \frac{[n]_{q}[n+1]_{q}}{[n+2]_{q}[n+3]_{q}}.
\end{equation*}
\end{itemize}
\end{enumerate}


\begin{Remark} There is another $q$-analogue of the sum of the first $n$ cubes given by Garrett and Hummel~\cite[Equation 2]{GH2004}. This can also be obtained by telescoping (take $u_k=(1-q^{k+2})$ and $v_k=-(1-q^k)$ in Lemma~\ref{E-telescoping-lemma}). Their elliptic extensions are immediate and are not included here. Further such $q$-analogues are obtained by Cigler~\cite{Cigler2014}, again by telescoping. 

\end{Remark}

Now, instead of taking $m=1$ in \eqref{tel-a}, we
take $m=2$, shift the index $k\mapsto k-1$ in \eqref{tel-a}
and replace $n$ by $n-1$. We then obtain
\begin{align}\label{m3rising2}
\sum_{k=1}^{n}  &W_{a,b;q,p}(k-1)[3]_{aq^{2k-2},bq^{k-1};q,p} [k]_{a,b;q,p} [k+1]_{a,b;q,p}\cr
&= [n]_{a,b;q,p}[n+1]_{a,b;q,p}[n+2]_{a,b;q,p}.
\end{align}

\subsection*{Some special cases of \eqref{m3rising2} }
We note some special cases of the $a;q$-special case of \eqref{m3rising2}
(which is obtained by first letting $p\to 0$, followed by letting $b\to 0$
in \eqref{m3rising2}), i.e.,
\begin{equation}\label{m3rising2a}
\sum_{k=1}^{n}  W_{a;q}(k-1)[3]_{aq^{2k-2};q} [k]_{a;q} [k+1]_{a;q}
= [n]_{a;q}[n+1]_{a;q}[n+2]_{a;q}.
\end{equation}
\begin{itemize}
\item Take ${a \to 0}$ in \eqref{m3rising2a} to obtain
\begin{equation*}
\sum_{k=1}^{n} q^{3n-3k}[k]_{q}[k+1]_{q}
=\frac{[n]_{q}[n+1]_{q}[n+2]_{q}}{[3]_{q}}.
\end{equation*}

\item Take ${a \to 1}$ in \eqref{m3rising2a} to obtain
\begin{equation*}
\sum_{k=1}^{n} q^{3n-3k}([k]_{q}[k+1]_{q})^2[2k+1]_{q}
=\frac{([n]_{q}[n+1]_{q}[n+2]_{q})^2}{[3]_{q}}.
\end{equation*}
\item Next, take ${a \to q}$ in \eqref{m3rising2a} to obtain
\begin{equation*}
\sum_{k=1}^{n} q^{3n-3k}[k]_{q}[k+1]_{q}^2[k+2]_{q}[2k+2]_{q}
=\frac{[n]_{q}[n+1]_{q}^2[n+2]_{q}^2[n+3]_{q}}{[3]_{q}}.
\end{equation*}
\item The following pair of identities is obtained by first
  replacing $q$ by $q^2$ and then letting $a\to q$, respectively, 
$a \to q^{-1}$:
\begin{align*}
\sum_{k=1}^{n} &q^{6n-6k}[2k]_{q}[2k+1]_{q}[2k+2]_{q}[2k+3]_{q}[4k+3]_{q}\\
&=\frac{[2n]_{q}[2n+1]_{q}[2n+2]_{q}[2n+3]_{q}[2n+4]_{q}[2n+5]_{q}} {[6]_{q}}.
\end{align*}
\begin{align*}
\sum_{k=1}^{n}& q^{6n-6k}[2k-1]_{q}[2k]_{q}[2k+1]_{q}[2k+2]_{q}[4k+1]_{q}\\
&=\frac{[2n-1]_{q}[2n]_{q}[2n+1]_{q}[2n+2]_{q}[2n+3]_{q}[2n+5]_{q}} {[6]_{q}}.
\end{align*}

\end{itemize}

Finally, before closing this section, we note the elliptic extension of \eqref{rising-fact2}.

\begin{Theorem} For $n, m$ non-negative integers, we have, 
\begin{align}
\sum_{k=1}^n  &\frac{W_{a,b;q,p}(k)\, [m]_{aq^{2k},bq^k;q,p}}
{ [k]_{a,b;q,p}[k+1]_{a,b;q,p}\dots [k+m]_{a,b;q,p}} \cr
&= \bigg(\frac{1}{[m]_{a,b;q,p}! }-\frac{1}{[n+1]_{a,b;q,p}[n+2]_{a,b;q,p}\dots [n+m]_{a,b;q,p}}\bigg),\label{tel-b}
\end{align}
where $[m]_{a,b;q,p}! :=[m]_{a,b;q,p}[m-1]_{a,b;q,p}\cdots[1]_{a,b;q,p}$
is an elliptic analogue of the factorial of $m$.
\end{Theorem}
\begin{proof}
We apply Lemma~\ref{E-telescoping-lemma} and take
\begin{align*}
u_k &=  \frac{1}{[k+2]_{a,b;q,p}[k+3]_{a,b;q,p}\dots [k+m+1]_{a,b;q,p}},\\
v_k = u_{k-1} &= \frac{1}{[k+1]_{a,b;q,p}[k+2]_{a,b;q,p}\dots [k+m]_{a,b;q,p}};\\
\intertext{so that} 
t_k &= \frac{-W_{a,b;q,p}(k+1) [m]_{aq^{2k+2},bq^{k+1};q,p}}{ [k+1]_{a,b;q,p}[k+2]_{a,b;q,p}\dots [k+m+1]_{a,b;q,p}
}
\text{ and } 
t_0 &= \frac{-W_{a,b;q,p}(1)[m]_{aq^2,bq;q,p}}{[m+1]_{a,b;q,p}!}.
\end{align*}
With these substitutions, we have \eqref{eq:E-telescoping-lemma}, and
after replacing $n$ by $n-1$ and shifting the index of the sum (such that
$k$ runs from $1$ to $n$, instead of from $0$ to $n-1$)
we readily obtain \eqref{tel-b}.
\end{proof}

%
%

\section{Special cases of elliptic and multibasic hypergeometric series identities}\label{sec3}
 In \cite{MS2004}, 
 the indefinite summation formula
\begin{equation}
\sum_{k=0}^{n} \frac{(1-aq^{2k})}{(1-a)} \frac{\qrfac{a, b}{k}}{\qrfac{q, aq/b}{k}} b^{n-k} = \frac{\qrfac{aq, bq}{n}}{\qrfac{q, aq/b}{n}}\label{indef-sum-1}
\end{equation}
is used to obtain $q$-analogues of several elementary sums. This includes Warnaar's \cite{SOW2004} $q$-analogue of the sum of the first $n$ cubes. 
In this section, we use the same idea, but use the following elliptic analogue of \eqref{indef-sum-1}:
\begin{align}\label{E-indef-sum-1}
\sum_{k=0}^{n} \frac{\theta(aq^{2k};p^2)}{\theta(a;p^2)} \frac{\pqrfac{a,b,cp}{k}{q;p^2}}{\pqrfac{q, aq/b, bcpq}{k}{q,p^2}} \frac{\pqrfac{bcp/a}{k}{q^{-1},p^2}}{\pqrfac{cp/aq}{k}{q^{-1},p^2}} b^{n-k}\cr
= \frac{\pqrfac{aq,bq,cpq}{n}{q,p^2}}{\pqrfac{q, aq/b, bcpq}{n}{q,p^2}}\frac{\pqrfac{bcp/aq}{n}{q^{-1},p^2}}{\pqrfac{cp/aq}{n}{q^{-1},p^2}}.
\end{align}

We first give some remarks before the proof. Clearly, 
\eqref{E-indef-sum-1} reduces to \eqref{indef-sum-1} when ${p = 0}$. We cannot take $c=0$ in \eqref{E-indef-sum-1} while keeping the nome $p^2$, as $c=0$
appears as an essential singularity on each side of \eqref{E-indef-sum-1}.
The extra parameter $c$ ensures that the elliptic balancing condition
holds for the terms appearing in \eqref{indef-sum-1}. The way
the $q$-series identity \eqref{indef-sum-1} is extended to the elliptic
identity in \eqref{E-indef-sum-1} is analogous to the way the
of the $q$-Saalsch\"utz summation is extended to the elliptic case
as described in \cite[Sec.\ 11.4, p.~323]{GR90}.
Notice that the indefinite summation
\eqref{E-indef-sum-1} can also be obtained by telescoping (just as
\eqref{indef-sum-1}).

\begin{proof}[Proof of \eqref{E-indef-sum-1}]
  A direct way to obtain \eqref{E-indef-sum-1} is to deduce it from
  the Frenkel and Turaev $_{10}V_9$ summation \cite[Eq.\ (11.4.1)]{GR90},
  which is an elliptic analogue of Jackson's very-well-poised
  $_{8}\phi_{7}$ summation.
  Specifically, taking ${e \to aq^{n+1}}$ in \cite[Equation (11.4.1)]{GR90}
  we obtain 
\begin{equation*}
\sum_{k=0}^{n} \frac{\theta(aq^{2k};p)}{\theta(a;p)} \frac{\pqrfac{a,b,c,a/bc}{k}{q,p}}{\pqrfac{q, aq/b, aq/c, bcq}{k}{q,p}} q^k = \frac{\pqrfac{aq,bq,cq,aq/bc}{n}{q,p}}{\pqrfac{q, aq/b, aq/c, bcq}{n}{q,p}}.
\end{equation*}
Now replace $p$ by $p^2$ and subsequently replace $c$  by $cp$ 
and use
\begin{equation*}
\frac{\pqrfac{a/bcp}{k}{q,p^2}}{\pqrfac{aq/cp}{k}{q,p^2}}
= \frac{1}{b^kq^k}\frac{\pqrfac{bcp/a}{k}{q^{-1},p^2}}{\pqrfac{cp/aq}{k}{q^{-1},p^2}}.
\end{equation*}
This immediately gives \eqref{E-indef-sum-1}.
\end{proof} 


It is easy to use \eqref{E-indef-sum-1} to obtain elliptic extensions
of results from \cite{MS2004}. However, these results necessarily
have the additional parameter $c$, which cannot be specialized to
$0$ or $\infty$ before letting $p=0$. As an example,
we give another extension of Warnaar's result in \cite[Equation~(2)]{SOW2004},
which is a $q$-analogue of the sum of cubes. 

Replace $n$ by $n-1$, shift the index of summation $k \to k-1$,
and set $a=b=q^2$ in \eqref{E-indef-sum-1} to obtain:
\begin{align*}
\sum_{k=1}^{n} \frac{\theta(q^{2k};p^2)}{\theta(q^2;p^2)} \frac{\pqrfac{q^2, q^2, cp}{k-1}{q;p^2}}{\pqrfac{q, q, cpq^3}{k-1}{q,p^2}} \frac{\pqrfac{cp}{k-1}{q^{-1},p^2}}{\pqrfac{cp/q^3}{k-1}{q^{-1},p^2}} q^{2(n-k)}\cr
= \frac{\pqrfac{q^3, q^3, cpq}{n-1}{q,p^2}}{\pqrfac{q, q,  cpq^3}{n-1}{q,p^2}}\frac{\pqrfac{cp/q}{n-1}{q^{-1},p^2}}{\pqrfac{cp/q^3}{n-1}{q^{-1},p^2}}.\cr
\end{align*}
When $p=0$, this reduces to \eqref{warnaar-cubes}.

\begin{Remark}
A special case of \eqref{indef-sum-1} is the 
following $q$-analogue of the formula for the sum of the first $n$ odd numbers
(cf.\ \cite[Equation (3.9)]{MS2004}):
\begin{equation}\label{qodds}
\sum_{k=0}^{n-1}[2k+1]_qq^{-k}
=[n]_q^2q^{1-n}.
\end{equation}
An extension of \eqref{qodds} to cubic basic hypergeometric series
can be given as follows:
\begin{equation}
\sum_{k=0}^{n-1} q^{-k}
\frac{(aq;q^3)_k}{(aq^5;q^3)_k}
\frac{(1-q^{2k+1})}{1-q}\frac{(1-aq^{2k+1})^2}{(1-aq)^2}
=
\frac{(1-q^n)^2(1-aq^n)}{(1-q)^2(1-aq)}
\frac{\pqrfac{aq^4}{n-1}{q^3}}{\pqrfac{aq^5}{n-1}{q^3}}
q^{1-n}.
\end{equation}
(For $a\to 0$ this reduces to \eqref{qodds}.)
It is easy to verify that this sum telescopes. 
\end{Remark}

\begin{Remark}
  Another general indefinite elliptic summation that can
  be specialized to obtain various extensions of classical results
  is the following
  special case of a multibasic theta function identity by Gasper and
  Schlosser~\cite[Equation (3.19), $t=q$]{GS2005}:
\begin{align}\label{m00}
&\sum^n_{k=0}
\frac{\ta (ad(rs)^k, br^k/dq^k, cs^k/dq^k;p)}
{\ta (ad, b/d, c/d;p)}\cr
&\quad\times\frac{(ad^2/bc;q,p)_k  (b;r,p)_k (c;s,p)_k (a;rs/q,p)_k
}
{(dq;q,p)_k (adr/c;r,p)_k (ads/b;s,p)_k 
(bcrs/dq;rs/q,p)_k}q^k\cr
&=\frac{\ta(a, b, c, ad^2/bc;p)}{d\,\ta(ad, b/d, c/d, ad/bc;p)}\cr
&\quad\times
\frac{ (ad^2q/bc;q,p)_n (br;r,p)_n (cs;s,p)_n (ars/q;rs/q,p)_n}
{(dq;q,p)_n  (adr/c;r,p)_n (ads/b;s,p)_n (bcrs/dq;rs/q,p)_n}\cr
&\quad -\frac{\ta(d, ad/b, ad/c, bc/d;p)}
{d\,\ta(ad, b/d, c/d, ad/bc;p)}.
\end{align}
\end{Remark}

\section{The proof of Theorem~\ref{big-theorem} and some special cases}\label{sec4}
We have seen that telescoping leads to several elementary identities. All the telescoping identities are special cases of 
Euler's telescoping lemma, Lemma~\ref{E-telescoping-lemma}. In order to apply the telescoping lemma, we would like to use sequences
$u_k$, $v_k$ such that $t_k=u_k-v_k$ can be simplified. 

We now turn to the proof of Theorem~\ref{big-theorem}.
The motivation behind this theorem is to use \eqref{quad-rel} so that $t_k=u_k-v_k$ becomes an analogue of a factorized product of linear factors  in $k$. 


\begin{proof}[Proof of Theorem~\ref{big-theorem}]
 We combine Lemma~\ref{E-telescoping-lemma} with a special instance of the
  the difference equation \eqref{quad-rel}. 
Let  
\begin{align*}
x&=(gk-g+c)(hk+d),\\
y&=-(gk+c)(hk-h+d),\\
r&=2ghk+ch+dg.
\end{align*}
With this assignment of variables, we have $x+r=(gk+c)(hk+h+d)$,
$y-r=-(gk+g+c)(hk+d)$, and $r+x-y=2(gk+c)(hk+d)$.
Substituting these values into \eqref{quad-rel}, we have
\begin{align*}
&\big[(gk-g+c)(hk+d)\big]_{a,b;q,p}\,
\big[-(gk+c)(hk-h+d)\big]_{aq^{4(gk+c)(hk+d)},bq^{2gk+c)(hk+d)};q,p}\\
&-\big[(gk+c)(hk+h+d)\big]_{a,b;q,p}\,
\big[-(gk+g+c)(hk+d)\big]_{aq^{4(gk+c)(hk+d)},bq^{2(gk+c)(hk+d)};q,p}\\
&=\big[2(gk+c)(hk+d)\big]_{a,b;q,p}\,
\big[2ghk+ch+dg\big]_{aq^{2(gk-g+c)(hk+d)},bq^{(gk-g+c)(hk+d)};q,p}\\
&\quad\;\times
W_{aq^{4(gk+c)(hk+d)},bq^{2(gk+c)(hk+d)};q,p}\big(-(gk+g+c)(hk+d)\big),
\end{align*}
which is equivalent to
\begin{align*}
&-\big[(gk-g+c)(hk+d)\big]_{a,b;q,p}\,
\big[(gk+c)(hk-h+d)\big]_{aq^{2(gk+c)(hk+h+d)},bq^{(gk+c)(hk+h+d)};q,p}\\
&\quad\;\times
W_{aq^{2(gk+c)(hk+h+d)},bq^{(gk+c)(hk+h+d)};q,p}
\big((gk+c)(hk-h+d)\big)^{-1}\\
&+\big[(gk+c)(hk+h+d)\big]_{a,b;q,p}\,
\big[(gk+g+c)(hk+d)\big]_{aq^{2(gk-g+c)(hk+d)},bq^{(gk-g+c)(hk+d)};q,p}\\
&\quad\;\times
W_{aq^{2(gk-g+c)(hk+d)},bq^{(gk-g+c)(hk+d)};q,p}
\big((gk+g+c)(hk+d)\big)^{-1}\\
&=\big[2(gk+c)(hk+d)\big]_{a,b;q,p}\,
\big[2ghk+ch+dg\big]_{aq^{2(gk-g+c)(hk+d)},bq^{(gk-g+c)(hk+d)};q,p}\\
&\quad\;\times
W_{aq^{2(gk-g+c)(hk+d)},bq^{(gk-g+c)(hk+d)};q,p}\big((gk+g+c)(hk+d)\big)^{-1}.
\end{align*}
Multiplication of both sides of this relation by the factor
$$W_{aq^{2(gk-g+c)(hk+d)},bq^{(gk-g+c)(hk+d)};q,p}\big((gk+g+c)(hk+d)\big)$$
and application of the reduction
\begin{align*}
&\frac{W_{aq^{2(gk-g+c)(hk+d)},bq^{(gk-g+c)(hk+d)};q,p}\big((gk+g+c)(hk+d)\big)}
{W_{aq^{2(gk+c)(hk+h+d)},bq^{(gk+c)(hk+h+d)};q,p}\big((gk+c)(hk-h+d)\big)}\\
&=W_{aq^{2(gk-g+c)(hk+d)},bq^{(gk-g+c)(hk+d)};q,p}\big(2ghk+ch+dg\big)
\end{align*}
gives the identity
\begin{align}\label{eq:telescope}
&\big[(gk+c)(hk+h+d)\big]_{a,b;q,p}\,
\big[(gk+g+c)(hk+d)\big]_{aq^{2(gk-g+c)(hk+d)},bq^{(gk-g+c)(hk+d)};q,p}\notag\\
&-\big[(gk-g+c)(hk+d)\big]_{a,b;q,p}\,
\big[(gk+c)(hk-h+d)\big]_{aq^{2(gk+c)(hk+h+d)},bq^{(gk+c)(hk+h+d)};q,p}\notag\\
&\quad\;\times
W_{aq^{2(gk-g+c)(hk+d)},bq^{(gk-g+c)(hk+d)};q,p}\big(2ghk+ch+dg\big)\notag\\
&=\big[2(gk+c)(hk+d)\big]_{a,b;q,p}\,
\big[2ghk+ch+dg\big]_{aq^{2(gk-g+c)(hk+d)},bq^{(gk-g+c)(hk+d)};q,p}.
\end{align}
Thus, in order to apply Lemma~\ref{E-telescoping-lemma}, we let
\begin{align*}
t_k&=\big[2(gk+c)(hk+d)\big]_{a,b;q,p}\,
\big[2ghk+ch+dg\big]_{aq^{2(gk-g+c)(hk+d)},bq^{(gk-g+c)(hk+d)};q,p},\\
u_k&=\big[(gk+c)(hk+h+d)\big]_{a,b;q,p}\,
\big[(gk+g+c)(hk+d)\big]_{aq^{2(gk-g+c)(hk+d)},bq^{(gk-g+c)(hk+d)};q,p},\\
v_k&=\big[(gk-g+c)(hk+d)\big]_{a,b;q,p}\,
\big[(gk+c)(hk-h+d)\big]_{aq^{2(gk+c)(hk+h+d)},bq^{(gk+c)(hk+h+d)};q,p}\\
&\quad\;\times
W_{aq^{2(gk-g+c)(hk+d)},bq^{(gk-g+c)(hk+d)};q,p}\big(2ghk+ch+dg\big).
\end{align*}
Now by \eqref{eq:telescope} we have $t_k=u_k-v_k$, and \eqref{eq:E-telescoping-lemma}
gives the desired result.
\end{proof}

%

\subsection*{Some special cases of \eqref{bigid}}
\begin{enumerate}
\item An  $a;q$-analogue: take $p \to 0$ and  $b \to 0$.
\begin{multline}\label{spc-1}
\sum_{k=0}^n
\bigg(\frac{[2(gk+c)(hk+d)]_{a;q}[2ghk+ch+dg]_{aq^{2(gk-g+c)(hk+d)};q}}{[2cd]_{a;q}[ch+dg]_{aq^{2(c-g)d};q}}\\
\hspace{-1cm}\times
\prod_{j=0}^{k-1} \frac{[(gj+g+c)(hj+d)]_{aq^{2(gj-g+c)(hj+d)};q}}{[(gj+g+c)(hj+d)]_{aq^{2(gj+g+c)(hj+2h+d)};q}}\\ 
\times \prod_{j=0}^{k-1} W_{aq^{2(gj+c)(hj+h+d)};q}(2ghj+2gh+ch+dg)^{-1}\bigg)\\ 
= \frac{[(gn+c)(hn+h+d)]_{a;q}[(g+c)d]_{aq^{2(c-g)d};q}}{[2cd]_{a;q}[ch+dg]_{aq^{2(c-g)d};q}}
 \prod_{j=1}^{n} \frac{[(gj+g+c)(hj+d)]_{aq^{2(gj-g+c)(hj+d)};q}}{[(gj+c)(hj-h+d)]_{aq^{2(gj+c)(hj+h+d)};q}}\\
\hspace{-2cm}\times \prod_{j=1}^{n} W_{aq^{2(gj-g+c)(hj+d)};q}(2ghj+ch+dg)^{-1}\\
 - \frac{[(c-g)d]_{a;q}[c(d-h)]_{aq^{2c(h+d)};q}}{[2cd]_{a;q}[ch+dg]_{aq^{2(c-g)d};q}}
W_{aq^{2(c-g)d};q}(ch+dg).
\end{multline}
\item A $q$-analogue.
Take $a \to 0$ in \eqref{spc-1}. 
\begin{multline}\label{spc-2}
\sum_{k=0}^n \bigg(  
\frac{[2(gk+c)(hk+d)]_{q}[2ghk+ch+dg]_{q}}{[2cd]_{q}[ch+dg]_{q}}
q^{-\big( gh k^2+(ch+dg+gh)k\big)}
 \bigg)\\
= \frac{[(gn+c)(hn+h+d)]_{q}[(gn+g+c)(hn+d)]_{q}}{[2cd]_{q}[ch+dg]_{q}} q^{-\big(ghn^2 +\left(ch+dg+gh\right)n\big)}\\
-\frac{[c(d-h)]_{q}[(c-g)d]_{q}}{[2cd]_{q}[ch+dg]_{q}}q^{ch+dg}.
\end{multline}

\item 
We can further specialize $c,d,g$ and $h$ in \eqref{spc-2} to obtain more $q$-analogues, highlighted in \S\ref{intro}. In particular, we have the following:
\begin{enumerate}
\item Take $c,d,g \to 1$ and $h \to 0$, shift the index to run from $k=1$ to $n+1$, and replace $n+1$ by $n$ to obtain \eqref{spc-4(i)}.
\item Take $c,d,g,h \to 1$ to get \eqref{spc-4(ii)}.
\end{enumerate}
\end{enumerate}

\section*{Acknowledgements}
The research of Michael J.\ Schlosser was partially supported by the
Austrian Science Fund (FWF), grant~P~32305.
%
%



	
\end{document}